\documentclass{article}

\usepackage{amsmath}
\usepackage{amsthm}
\usepackage{amssymb}

\usepackage[colorlinks=true,urlcolor=blue,linkcolor=blue,plainpages=false,pdfpagelabels
]{hyperref}

\newcommand{\details}[1]{$\,$\\***DETAILS {\bf #1}***\\}
\renewcommand{\details}[1]{$\,$\\***DETAILS {\bf #1}***\\}
\renewcommand{\details}[1]{}

\newtheorem{theorem}{Theorem}[section]

\newtheorem{lemma}[theorem]{Lemma}

\newcommand{\be}{\begin{enumerate}}
\newcommand{\ee}{\end{enumerate}}
\newcommand{\beq}{\begin{equation}}
\newcommand{\eeq}{\end{equation}}
\newcommand {\bea} {\begin{eqnarray}}
\newcommand {\eea} {\end {eqnarray}}
\newcommand {\bua} {\begin{eqnarray*}}
\newcommand {\eua} {\end {eqnarray*}}

\newcommand{\se}{\subseteq}
\newcommand{\eps}{\varepsilon}
\def\N{{\mathbb N}}
\newcommand{\ds}{\displaystyle}
\newcommand{\limn}{\ds\lim_{n\to\infty}}
\newcommand{\lan}{\left\langle}
\newcommand{\ran}{\right\rangle}

\begin{document}

\title{A rate of asymptotic regularity for the Mann iteration of $\kappa$-strict pseudo-contractions}
\author{Daniel Ivan${}^{1,} \footnote{The research of Daniel Ivan was done while he was a member of the project  PN-II-ID-PCE-2011-3-0383 at  Simion Stoilow Institute of Mathematics of the Romanian Academy.} $ , Lauren\c{t}iu Leu\c{s}tean${}^{2,3}$,  \\[0.2cm]
\footnotesize ${}^1$ 17 Lumbervale Ave, Toronto M6H 1C7, Ontario, Canada\\
\footnotesize ${}^2$ Faculty of Mathematics and Computer Science, University of Bucharest,\\
\footnotesize Academiei 14,  P.O. Box 010014, Bucharest, Romania\\[0.1cm]
\footnotesize ${}^3$ Simion Stoilow Institute of Mathematics of the Romanian Academy, \\
\footnotesize P.O. Box 1-764, RO-014700, Bucharest, Romania\\[0.1cm]
\footnotesize E-mails:  idaniel.ivan@gmail.com, Laurentiu.Leustean@imar.ro
}

\date{}

\maketitle

\begin{abstract}
\noindent In this paper we apply methods of proof mining to obtain a uniform effective rate of asymptotic regularity for the Mann 
iteration associated to $\kappa$-strict pseudo-contractions on convex subsets of Hilbert spaces.\\

\noindent {\em MSC:}  47J25; 47H09;  03F10.\\

\noindent {\em Keywords:} Proof mining; Effective bounds; Asymptotic regularity; 
Strict pseudo-contractions; Mann iteration
\end{abstract}

\maketitle 

\pagenumbering{arabic}

\section{Introduction}

Let $H$ be a real Hilbert space, $C\se X$ a nonempty closed convex subset, $T:C\to C$ be a mapping and $0\leq \kappa < 1$. $T$ is said to be a 
\emph{$\kappa$-strict pseudo-contraction} if for all $x,y\in C$,
\beq 
 \|Tx - Ty\|^2 \leq \|x-y\|^2 + \kappa \|x - Tx- \left(y -T y \right)\|^2. \label{def-k-strict-pseudo-contraction}
\eeq
This class of nonlinear mappings was introduced in the 60's by Browder and Petryshyn \cite{BroPet67}. 
Nonexpansive mappings  coincide with  $0$-strict pseudo-contrac-tions.

The Mann iteration \cite{Kra55,Man53,Gro72} starting with $x\in C$ is defined by 
\beq
x_0:=x, \quad x_{n+1}:=(1-\lambda_n)x_n+\lambda_nTx_n,
\eeq
where $(\lambda_n)$ is  a sequence in $(0,1)$. By letting $\lambda_n=\lambda$ for all $n\in\N$, 
we get the Krasnoselskii iteration \cite{Kra55} as  a special case. In the sequel, we consider a sequence
$(\lambda_n)$ satisfying the following conditions:
\beq 
\kappa< \lambda_n < 1 \text{ for all }n\in\N \text{~~and~~}  \sum_{n=0}^\infty (\lambda_n - \kappa)(1-\lambda_n) = \infty. \label{lambdan-cond}
\eeq
Assuming that $T$ has fixed points and $(\lambda_n)$ satisfies \eqref{lambdan-cond}, 
Marino and Xu \cite{MarXu07} proved the weak convergence of the Mann iteration 
$(x_n)$ to a fixed point of $T$. Their result generalizes the one obtained by Browder and 
Petryshyn for the Krasnoselskii iteration. 
Furthermore, as an immediate consequence one gets Reich's result \cite{Rei79} for nonexpansive 
mappings in Hilbert spaces.

As it is the case with many results on the weak or strong convergence of nonlinear iterations, 
the first step in their proofs consists in getting the {\em asymptotic regularity} of $(x_n)$, 
i.e. the fact that $\ds \limn \|x_n-Tx_n\|=0$ for all starting points $x\in C$. This is a very important 
property, introduced by Browder and Petryshyn \cite{BroPet66} for the Picard iteration $x_n=T^nx$. 
The following result is implicit in \cite{MarXu07}:

\begin{theorem} \label{thm:MX07:result}
Let $C$ be a convex subset of a Hilbert space $H$, $T:C\to C$ be a $\kappa$-strict 
pseudo-contraction such that $T$ has fixed points. Assume that $(\lambda_n)$ satisfies \eqref{lambdan-cond}.
Then  $\limn \|x_n-Tx_n\|=0$ for all $x\in C$.
\end{theorem}

In this paper we  apply proof mining methods to obtain a finitary, quantitative version of 
a generalization of Theorem \ref{thm:MX07:result}, computing a uniform rate of asymptotic regularity for 
the Mann iteration $(x_n)$, i.e. a rate of convergence of $(\|x_n-Tx_n\|)$ towards $0$. 
The fact that we can get such a result is guaranteed by logical metatheorems for Hilbert spaces
proved by Kohlenbach \cite{Koh05}. Moreover, as an immediate consequence of our main result, we obtain a 
quadratic rate of asymptotic regularity for the Krasnoselskii iteration.

\section{Main result}

Given $x\in X$ and $b,\delta>0$, we use the notation
\[Fix_{\delta}(T,x,b)=\{y\in C\mid \|x-y\|\le b \quad \text{and} \quad \|y-Ty\| <\delta\}\]
and we say that $T$ has {\em  approximate fixed points} in a $b$-neighborhood of $x$ if $Fix_{\delta}(T,x,b)\ne \emptyset$ for all $\delta>0$. 
If $T$ has a fixed point $p\in C$, then for all $x\in C$ and any $b\geq d(x,p)$, we have that $p\in Fix_{\delta}(T,x,b)$ for all $\delta>0$.

Let us recall that a rate of divergence for a divergent series  $\ds\sum_{n=0}^\infty a_n$ is a mapping $\theta:\N\to\N$  satisfying 
$\ds\sum_{k=0}^{\theta(n)}a_k\geq n$ for all $n\in\N$. \\

The main result of this paper is the following finitary, quantitative version of a generalization of Theorem \ref{thm:MX07:result}, where the hypothesis of $T$ having fixed points is weakened to the one that $T$ has approximate fixed points in a $b$-neighborhood of $x$ for some $x\in C$ and $b>0$.

\begin{theorem}\label{main-quant}
Let $H$ be a Hilbert space, $C\se H$ a nonempty convex subset and 
$T:C \to C$ be a $\kappa$-strict pseudo-contraction, where $0 \leq \kappa < 1$.
Assume that $(\lambda_n)$  is a sequence in $(\kappa,1)$  satisfying 
$\ds \sum_{n=0}^\infty (\lambda_n -\kappa)(1-\lambda_n) = \infty$ with rate of divergence $\theta:\N\to\N$.
Let $x\in C,b>0$ be such that $\|x-Tx\|\leq b$ and $T$ has approximate fixed points in a $b$-neighborhood of $x$.

Then  $\limn \|x_n-Tx_n\|=0$ and 
\beq
\forall \eps>0\forall n\ge \Phi(\varepsilon,b,\theta)\big(\|x_n-Tx_n\|<\eps\big), \quad \text{where}\quad \Phi(\eps,b,\theta)=\theta\left(\left\lceil\frac{b^2}{\eps^2}\right\rceil \right).\label{def-Phi-main-quant}
\eeq
\end{theorem}
\begin{proof} See Section \ref{proof-main-quant}.\end{proof}

Browder and Petryshyn proved \cite{BroPet67} that if $C$ is bounded, then $T$ has fixed points. Hence, by letting $b$ to be an upper bound for the diameter $d_C$ of $C$, we get that $Fix_{\delta}(T,x,b)\ne \emptyset$ for all $x\in C$. As a consequence of our main theorem, for bounded $C$, the Mann iteration $(x_n)$ is asymptotically regular with rate of asymptotic regularity $\Phi$ given by \eqref{def-Phi-main-quant}, where $b\geq d_C$.

Furthermore, if $\lambda_n=\lambda\in (\kappa,1)$ then one can easily verify that 
$$\ds \theta(n):=\left\lceil\frac{1}{(\lambda-\kappa)(1-\lambda)}\right\rceil n$$ 
is a rate of divergence for the series $\ds \sum_{n=0}^\infty (\lambda -\kappa)(1-\lambda) = \infty$ . Hence, we get in this case that 
\beq
\Phi(\eps,b,\lambda,\kappa)=\left\lceil\frac{1}{(\lambda-\kappa)(1-\lambda)}\right\rceil\left\lceil\frac{b^2}{\eps^2}\right\rceil
\eeq
is a quadratic in $1/\eps$ rate of asymptotic regularity for the Krasnoselskii iteration $(x_n)$.\\

Since $0$-strict pseudo-contractions coincide with nonexpansive mappings, our results generalize with slightly changed bounds the ones obtained  by Kohlenbach \cite{Koh03} for the Mann iteration, and Browder and Petryshyn \cite{BroPet67} for the Krasnoselskii iteration associated to nonexpansive mappings in Hilbert spaces. 
We point out that in \cite{Koh03}, Kohlenbach computes, applying also proof mining methods, rates of asymptotic regularity for the Mann iteration of a nonexpansive mapping in the more general class of uniformly convex Banach spaces, generalized further by the second author to a class of uniformly convex geodesic spaces \cite{Leu07}.

\section{Some useful lemmas}

In the sequel, $H$ is a Hilbert space, $C\se H$ is a nonempty convex subset and 
$T: C \to C$ is a $\kappa$-strict pseudo-contraction. Furthermore, $(\lambda_n)$ 
is a sequence in $(0,1)$ and $(x_n)$ is the Mann iteration starting with $x\in C$, defined 
by \eqref{def-k-strict-pseudo-contraction}.

The following identities in Hilbert spaces will be used in the sequel.

\begin{lemma}\label{lem:props:norm}
Let $x,y \in H$ and $t\in [0,1]$. Then
\begin{align*}
 \|x+y\|^2=\|x\|^2+\|y\|^2+2\lan x,y\ran \text{~and~}   \|x-y\|^2=\|x\|^2+\|y\|^2-2\lan x,y\ran\\
\|tx+(1-t)y\|^2 = t\|x\|^2 + (1-t)\|y\|^2 - t (1-t)\|x-y\|^2. 
\end{align*}
\end{lemma}

\begin{lemma}\label{lemma-ineq}
\be
\item\label{ineq-main-Tzy} For all $y,z\in C$,
\[
 \|Tz - y\|^2 \leq \|z-y\|^2 +\kappa\|z-Tz\|^2+(\kappa+1)\|y-Ty\|^2 +2\|z-Ty\|\|y-Ty\|. 
 \]
\item\label{xn1y-xny-xnTxn} For all $y\in C$ and all $n\geq 0$
\begin{align*}
 \|x_{n+1} - y\|^2 \leq&  \|x_n-y\|^2-(\lambda_n-\kappa) (1-\lambda_n)\|x_n - Tx_n\|^2\\
 & + 2\|y-Ty\|(\|x_n-y\|+2\|y-Ty\|).
\end{align*}
\ee
\end{lemma}
\begin{proof}
\be
\item 
\begin{align*}
 \|Tz - y\|^2 =& \|(Tz-Ty)+(Ty-y)\|^2 \\
 = & \|Tz-Ty\|^2 + \|Ty-y\|^2+2\lan Tz-Ty,Ty-y\ran\\
\leq & \|z-y\|^2 +\kappa\|(z-Tz)-(y-Ty)\|^2+ \|Ty-y\|^2\\
& +2\lan Tz-Ty,Ty-y\ran\\
=& \|z-y\|^2 +\kappa\|z-Tz\|^2+(\kappa+1)\|y-Ty\|^2-2\lan z-Tz,y-Ty\ran\\
& +2\lan Tz-Ty,Ty-y\ran\\
=&  \|z-y\|^2 +\kappa\|z-Tz\|^2+(\kappa+1)\|y-Ty\|^2 +2\lan z-Ty,Ty-y\ran\\
\leq &  \|z-y\|^2 +\kappa\|z-Tz\|^2+(\kappa+1)\|y-Ty\|^2 +2\|z-Ty\|\|y-Ty\|
\end{align*}
\item  
\begin{align*}
\|x_{n+1} - y\|^2  =& \|\lambda_n x_n \! +\! (1-\lambda_n) Tx_n -y\|^2 \!=\!\|\lambda_n (x_n  - y) \!+ \!(1-\lambda_n) (Tx_n - y) \|^2\\
=& \lambda_n \|x_n  - y\|^2 +(1-\lambda_n) \|Tx_n -y\|^2  - \lambda_n (1-\lambda_n)\|x_n - Tx_n\|^2 \\
\leq &  \lambda_n \|x_n  - y\|^2 +(1-\lambda_n)\|x_n  - y\|^2+(1-\lambda_n)\kappa \|x_n - Tx_n\|^2\\
& +(1-\lambda_n)(\kappa+1)\|y-Ty\|^2 + 2(1-\lambda_n)\|x_n-Ty\|\|y-Ty\| \\
&  -\lambda_n (1-\lambda_n)\|x_n - Tx_n\|^2 \quad \text{ by \eqref{ineq-main-Tzy}}\\
=& \|x_n-y\|^2-(\lambda_n-\kappa) (1-\lambda_n)\|x_n - Tx_n\|^2\\
&+ (1-\lambda_n)(\kappa+1)\|y-Ty\|^2+2(1-\lambda_n)\|x_n-Ty\|\|y-Ty\|\\
\leq & \|x_n-y\|^2-(\lambda_n-\kappa) (1-\lambda_n)\|x_n - Tx_n\|^2\\
&+2\|y-Ty\|(\|x_n-y\|+2\|y-Ty\|),
\end{align*}
since $\|x_n-Ty\|\leq \|x_n-y\|+\|y-Ty\|$. 
\ee
\end{proof}

In particular, if $p$ is a fixed point of $T$, then for all $n\ge 0$, 
\beq
 \|x_{n+1} - p\|^2 \leq \|x_n - p\|^2-(\lambda_n-\kappa) (1-\lambda_n)\|x_n - Tx_n\|^2. \label{xn1p-xnp-xnTxn}
\eeq

A very important property of the Mann iteration is the following one

\begin{lemma}\cite{MarXu07}\label{xnTxn-nonincreasing}
The sequence $(\|x_n - Tx_n\|)$ is nonincreasing. 
\end{lemma}

\begin{lemma}\label{afp-b-upper-bound}
Let $y\in C$ and $b\geq \max\{\|x-Tx\|, \|x-y\|\}$ and $c\geq \|y-Ty\|$. Then for all $n\geq 0$, 
\be
\item\label{afp-b-upper-bound-i} $\|x_n-y\|\leq (n+1)b$ and $\|Tx_n-y\|\leq (n+2)b$.
\item\label{afp-b-upper-bound-ii} $\|x_{n+1} - y\|^2 \leq \|x_n-y\|^2-
(\lambda_n-\kappa) (1-\lambda_n)\|x_n - Tx_n\|^2+2\big((n+1)b+2c\big)\|y-Ty\|$.
\ee
\end{lemma}
\begin{proof}
\be
\item 
By induction on $n$, taking into account that, for all $n$, we have that  $\|x_{n+1}-y\|\leq 
\lambda_n\|x_n-y\|+(1-\lambda_n)\|Tx_n-y\|$ and that 
$\|x_n-Tx_n\|\leq \|x-Tx\| \leq b$ , by Lemma \ref{xnTxn-nonincreasing}.
\details{$n=0$: $\|x_0-y\|=\|x-y\|\leq b$ and $\|Tx_0-y\|=\|Tx-y\|\leq \|Tx-x\|+\|x-y\|\leq 2b$.
$n\Ra n+1$: $\|x_{n+1}-y\|\leq \lambda_n\|x_n-y\|+(1-\lambda_n)\|Tx_n-y\|\leq \lambda_nnb+
(1-\lambda_n)(n+1)b\leq \lambda_n(n+1)b+(1-\lambda_n)(n+1)b=(n+1)b$ and
 $\|Tx_{n+1}-y\|\leq \|Tx_{n+1}-x_{n+1}\|+\|x_{n+1}-y\|\leq \|x-Tx\|+(n+1)b\leq (n+2)b$.}
 \item Apply \eqref{afp-b-upper-bound-i} and Lemma \ref{lemma-ineq}.\eqref{xn1y-xny-xnTxn}.
\ee
\end{proof}

\section{Proof of Theorem \ref{main-quant}}\label{proof-main-quant}

Let us denote, for simplicity, 
$\ds\Delta:=\sum_{n=0}^{\Phi} (\lambda_n-\kappa) (1-\lambda_n)\|x_n - Tx_n\|^2$. \\

{\bf Claim:} $\Delta\le b^2$.

{\bf Proof of claim:} We prove that $\Delta\le b^2+\sigma$ for all $\sigma\in(0,1)$. 
Apply the fact that $T$ has approximate fixed points in a $b$-neighborhood of $x$, and we get for 
$$\delta:=\frac{\sigma}{(\Phi+1)(\Phi b+2b+2)}$$
an $y\in C$ such that $\|x-y\|\le b$ and $\ds \|y-Ty\| <\delta <\frac12$. We can apply 
Lemma \ref{afp-b-upper-bound}.\eqref{afp-b-upper-bound-ii} with 
$b$ as in the hypothesis and $\ds c:=1/2$ to obtain
\begin{align*}
\Delta\leq & \sum_{n=0}^{\Phi} 
(\|x_{n}-y\|^2-\|x_{n+1}-y\|^2)+2\sum_{n=0}^{\Phi} \big((n+1)b+1\big)\|y-Ty\|\\
= & \|x_0-y\|^2-\|x_{\Phi+1}-y\|^2+2\|y-Ty\|\left(\frac{(\Phi+1)(\Phi+2)b}2+(\Phi+1)\right)\\
\leq & \|x-y\|^2+\|y-Ty\|(\Phi+1)(\Phi b+2b+2) < b^2+\sigma.  
\end{align*}
The claim is proved.  

Since $(\|x_n-Tx_n\|)$ is nonincreasing, it  is enough to prove that there exists $N\leq \Phi$ 
such that $\|x_N - Tx_N\|\leq \eps$. Assume by contradiction that for all $n=0,\ldots, 
\Phi$ one has $\|x_n-Tx_n\| >\eps$. It follows that 
\begin{align*}
\Delta=& \sum_{n=0}^{\Phi} (\lambda_n-\kappa) (1-\lambda_n)\|x_n - Tx_n\|^2 >  \sum_{n=0}^{\Phi} (\lambda_n-\kappa) (1-\lambda_n)
\eps^2 \\
= & \eps^2\sum_{n=0}^{\theta(\lceil b^2/\eps^2\rceil)} (\lambda_n-\kappa) (1-\lambda_n) \geq   \left\lceil\frac{b^2}{\eps^2}\right\rceil\cdot \eps^2= b^2.
\end{align*}
Thus, we have got a contradiction. \hfill $\Box$

\mbox{ } 

\noindent
{\bf Acknowledgements:} \\[1mm]
This work was supported by a grant of the Romanian National Authority for Scientific Research, CNCS - UEFISCDI, project 
number PN-II-ID-PCE-2011-3-0383.


\begin{thebibliography}{}



\bibitem{BroPet66} 
F.E. Browder, W.V. Petryshyn, The solution by iteration of nonlinear functional equations in Banach spaces, 
Bull. Amer. Math. Soc. 72 (1966), 571--575. 

\bibitem{BroPet67} 
F.E. Browder, W.V. Petryshyn, Construction of fixed points of nonlinear mappings in Hilbert spaces, 
J.  Math. Anal. Appl. 20 (1967), 197--228.

\bibitem{Gro72} 
C.W. Groetsch, A note on segmenting Mann iterates, J. Math. Anal. Appl. 40 (1972), 369--372.

\bibitem{Koh03}
U. Kohlenbach, Uniform asymptotic regularity for Mann iterates, J. Math. Anal. Appl. 279 (2003), 531--544.

\bibitem{Koh05} 
U. Kohlenbach, Some logical metatheorems with applications in functional analysis, 
Trans. Amer. Math. Soc. 357 (2005), 89-128.

\bibitem{Kra55} 
M. A. Krasnoselski, Two remarks on the method of successive approximation, Usp. Math. Nauk (N.S.) 10 (1955), 
123--127 (in Russian).

\bibitem{Leu07}
L. Leu\c{s}tean, A quadratic rate of asymptotic regularity for CAT(0)-spaces, J. Math. Anal. Appl. 325 (2007), 
386-399.


\bibitem{Man53} 
W.R. Mann, Mean value methods in iteration,  Proc. Amer. Math. Soc. 4 (1953), 506-510.

\bibitem{MarXu07} 
G. Marino, H.-K. Xu, Weak and strong convergence theorem for strict pseudo-contractions in Hilbert spaces, 
J. Math. Anal. Appl. 329 (2007), 336-346.


\bibitem{Rei79}
S. Reich, Weak convergence theorems for nonexpansive mappings in Banach spaces, 
J. Math. Anal. Appl. 67 (1979), 274--276.


\end{thebibliography}
\end{document}